\documentclass[12pt]{article}
\usepackage{amsfonts}

\usepackage{eurosym}
\usepackage{amsmath,amssymb,amscd,verbatim,amsthm}
\usepackage{latexsym, epsfig, graphpap, graphics, mathrsfs}
\usepackage[varg]{pxfonts}
\usepackage{mathrsfs}

\newtheorem{definition}{Definition}[section]
\newtheorem{theorem}{Theorem}[section]

\newtheorem{corollary}{Corollary}[section]

\input{tcilatex}

\begin{document}

\begin{center}
\textbf{Note on $f_\lambda$-statistical convergence}

\vskip 0.5 cm

Stuti Borgohain $^*${\footnote{%
The work of the first author was carried during her one month visit to
Istanbul Ticaret University, Turkey, in November'2015}} and Ekrem Sava$\c{s}$

Istanbul Commerce University, 34840 Istanbul, Turkey

E-mail : stutiborgohain$@$yahoo.com, ekremsavas@yahoo.com.
\end{center}

\vskip 1 cm

\noindent{\footnotesize \textbf{Abstract:} In this article, we study about
the $\lambda$-statistical convergence with respect to the density of moduli
and find some results related to statistical convergence as well. Also we
introduce the concept of $f_\lambda$-summable sequence and try to
investigate some relation between the $f_\lambda$-summability and module $%
\lambda$-statistical convergence.} \newline

\section{Introduction}

The idea of statistical convergence of real sequences is the extension of
convergence of real sequences. Initially the statistical convergence for
single sequences was introduced by Fast Fast \cite{Fast} in
1951  and  Schoenberg \cite{Sc}  in 1959, studied
statistical convergence as a summability method and listed some of
elementary properties of statistical convergence as well. Both of these
authors noted that if bounded sequence is statistically convergent, then it
is Cesaro summable. More investigations have been studied in the direction
of topological spaces, statistical Cauchy condition in uniform spaces,
Fourier analysis, ergodic theory, number theory, measure theory,
trigonometric series, turnpike theory and Banach spaces. Statistical
convergence turned out to be one of the most active areas of research in
summability theory after the works of Fridy \cite{Fridy} and Salat \cite
{Sat}. Recently, Savas \cite{savas1} introduced the generalized double
statistical convergence in locally solid Riesz spaces. Moreover, Mursaleen
\cite{Mursaleen} introduced the concept of $\lambda$-statistically
convergence by using the idea of $(V, \lambda)$-summability to generalize
the concept of statistical convergence. For more interesting investigations
concerning statistical convergence one may consult the papers of Cakalli \cite
{Ca}, Miller and Orhan \cite{Miler} and (\cite{Connor},  \cite{Kolk}, \cite{Kolk1},   \cite{Pa}, \cite{Sat},
\cite{savas11}, \cite{Tripathy}, \cite{Borgohain}) and others.\newline

The concept of density for sets of natural numbers with respect to the
modulus function was introduced by A. Aizpuru et al [1] in 2014. They
studied and characterized the generalization of this notion of $f$-density
with statistical convergence and proved that ordinary convergence is
equivalent to the module statistical convergence for every unbounded modulus
function. Also, in [2], A. Aizpuru and his team, they worked on double
sequence spaces for the results of $f$-statistical convergence by using
unbounded modulus function. The concept was further generalized and
characterized by Savas and Borgohain [\cite{Borgohain}, \cite{savas2}] .
\newline

The notion of statistical convergence depends on the idea of asymtotic
density of subsets of the set $\mathbb{N}$ of natural numbers. A subset $A$
of $\mathbb{N}$ is said to have natural density ${\delta(A)}$ if

\begin{equation*}
{\delta(A)}={\lim_{n \rightarrow \infty}}{\frac{1}{n}}{\sum\limits_{k=1}^n}{%
\chi_A}(k).
\end{equation*}

\noindent where $\chi_A$ is the characteristic function of $A$.\newline

We mean a sequence $(x_t)$ to be statistically convergent to $L$, if for any
$\xi > 0$, $\delta\{(t \in \mathbb{N}:\vert x_t - L \vert \geq \xi )\} = 0.$
Analogously, $(x_t)$ is said to be statistically Cauchy if for each $\xi>0$
and $n_0 \in \mathbb{N}$ there exists an integer $q \geq n_0$ such that $%
\delta(\{t \in \mathbb{N}:\Vert x_t - x_q \Vert < \xi \})=1.$\newline

Nakano \cite{Nakano} introduced the notion of a modulus function whereas
Ruckle \cite{Ruckle} and Maddox \cite{Maddox} have introduced and discussed
some properties of sequence spaces defined by using a modulus function. A
modulus function is a function $f: R^+ \rightarrow R^+ $ which satisfies:

\begin{enumerate}
\item $f(x)=0$ if and only if $x=0$.

\item $f(x+y) \leq f(x)+f(y)$ for every $x,y \in R^+$.

\item $f$ is increasing.

\item $f$ is continuous from the right at 0.
\end{enumerate}
Later on modulus function have been discussed in (\cite{Ba, savas8, savas9,savas10, savas12}) and others.

In this paper, we study the density on moduli with respect to the $\lambda$%
-statistical convergence. We also investigate some results on the new
concept of $f_\lambda$-statistical convergence with the ordinary
convergence. Also we find out some new concepts on $f_\lambda$-summability
theory and try to find out new results related to the $f_\lambda$-summable
sequences and $f_\lambda$-statistical convergent sequence.

\section{Definitions and basic results}

By density of moduli of a set $A \subseteq \mathbb{N}$, we mean $\delta_f(A)=%
\displaystyle\lim_u \frac{f(\vert A(u) \vert)}{f(u)}$, (in case this limit
exists )where $f$ is an unbounded modulus function. \newline

Let $(x_t)$ be a sequence in $X$ ($X$ is a normed space). If for each $\xi >
0$, $A=\{t \in \mathbb{N}: \Vert x_t-L\Vert > \xi \}$ has $f$-density zero,
then it is said that the $f$-statistical limit of $(x_t)$ is $L \in X$, and
we write it as $f$-stlim$x_t=L$. \newline

Observe that $\delta(A)=1-\delta(\mathbb{N} \backslash A)$.\newline

Let us assume that $A \subseteq \mathbb{N}$ and $A$ has $f$-density zero.
For every $u \in N$ we have $f(u) \leq f(\vert A(u)\vert)+f(\vert (\mathbb{N}
\backslash A)(u)\vert )$ and so

\begin{equation*}
1 \leq \frac{f(\vert A(u) \vert)}{f(u)} +\frac{f(\vert ( \mathbb{N}
\backslash A)(u)\vert )}{f(u)} \leq \frac{f(\vert A(u) \vert)}{f(u)}+1
\end{equation*}

By taking limits we deduce that $f$-density of $(\mathbb{N} \backslash A)$
is one.\newline

Let $\lambda=(\lambda_n)$ be a non-decreasing sequence of positive numbers
such that ,

\begin{equation*}
\lambda_1=1, \lambda_{n+1} \leq \lambda_n + 1 \mbox{~and~} \lambda_n
\rightarrow \infty \mbox{~as~} n \rightarrow \infty.
\end{equation*}

Note: The collection of all such sequences $\lambda$ will be denoted by $%
\Lambda$. \newline

A sequence $(x_t)$ of real numbers is said to be $\lambda$-statistically
convergent to $L$ if for any $\xi>0$,

\begin{equation*}
\displaystyle\lim_{n \rightarrow \infty} \frac{1}{\lambda_n} \vert \{ t \in
I_n :\vert x_t - L \vert \geq \xi \} \vert =0,
\end{equation*}

where $I_n=[n - \lambda_n+1, n ]$ and $\vert A \vert$ denotes the
cardinality of $A \subset N$ (refer \cite{Mursaleen}). \newline

More investigations in this direction and more applications of ideals can be found in \cite{savas1, savas4, savas5, savas13} where
 many important references can be found.

We define $f_\lambda$-statistical convergence as:

\begin{equation*}
\displaystyle\lim_{n \rightarrow \infty} \frac{f(\vert A(n) \vert )}{%
f(\lambda_n)} =0 \mbox{~ i.e.~} \delta_{f_\lambda}(A)=\displaystyle\lim_{n
\rightarrow \infty} \frac{f(\vert A(n) \vert )}{f(\lambda_n)}.
\end{equation*}

The set of all $f_{\lambda}$-statistically convergent sequences is denoted
by $S_{f_\lambda}$. \newline

A sequence $x=(x_t)$ is said to be strongly $f_\lambda$-summable to the
limit $L$ if

\begin{equation*}
\displaystyle\lim_n \frac{1}{f(\lambda_n)}\displaystyle\sum_{t \in I_n}
f(\vert x_t -L \vert)=0
\end{equation*}
and we write it as $x_t \rightarrow L[f_\lambda]$. In this case, $L$ is
called the $f_\lambda$-limit of $x$. We denote the class of all $f_\lambda$%
-summable sequences as $w_{f_\lambda}$.

\section{Main Results}

Maddox \cite{Maddox1} showed the existence of an unbounded modulus $f$ for
which there exists a positive constant $c$ such that $f(xy) \geq c f(x)f(y)$
for all $x \geq 0, y \geq 0$.\newline

\begin{theorem}
Let $f$ be an unbounded modulus such that there is a positive constant $c$
such that $f(xy) \geq cf(x)f(y)$ for all $x \geq 0, y \geq 0$ and $%
\displaystyle\lim_{u \rightarrow \infty} \frac{f(u)}{u} > 0$. A strongly $%
\lambda$-summable sequence $(x_t)$ is also a $f_\lambda$-statistically
convergent sequence. Moreover, for a bounded sequence $(x_t)$, $f_\lambda$%
-statistically convergence implies strongly $\lambda$-summability.
\end{theorem}

\begin{proof} By the definition of modulus function, we have for any sequence $x=(x_t)$ and $\xi >0$,
\begin{eqnarray*}
\displaystyle\sum_{t=1}^n f(\vert x_t - L \vert) & \geq & f\left(\displaystyle\sum_{t=1}^n f(\vert x_t - L \vert\right)\\
& \geq & f(\vert \{t \leq n: \vert x_t - L \vert \geq \xi \} \vert \xi )\\
& \geq & c f(\vert \{t \leq n: \vert x_t - L \vert \geq \xi \} \vert ) f(\xi )\\
\end{eqnarray*}

Since $x=(x_t)$ is strongly $\lambda$-summable sequence, so

\begin{eqnarray*}
\frac{1}{\lambda_n} \displaystyle\sum_{t=1}^n f(\vert x_t - L \vert) & \geq & \frac{c f(\vert \{t \leq n: \vert x_t - L \vert \geq \xi \} \vert ) f(\xi )}{\lambda_n}\\
& = & \frac{c f(\vert \{t \leq n: \vert x_t - L \vert \geq \xi \} \vert ) f(\xi ) f(\lambda_n)}{\lambda_n f(\lambda_n)}
\end{eqnarray*}

By using the fact that $\displaystyle\lim_{u \rightarrow \infty} \frac{f(u)}{u} > 0$ and $x$ is $\lambda$-statistical w.r.t. $f$, it can be summarized that $x$ is $f_\lambda$-statistcal convergent and this completes the proof of the theorem.
\end{proof}

\begin{theorem}
For a strongly $\lambda$-summable or statistically $f_\lambda$-convergent
sequence $x=(x_t)$ to $L$, there is a convergent sequence $y$ and a $%
f_\lambda$-statistically null sequence $z$ such that $y$ is convergent to $L$%
, $x=y+z$ and $\displaystyle\lim_n \frac{f(\vert \{ t \in I_n:z_t \neq 0 \}
\vert)}{f(\lambda_n)}=0$. Moreover, if $x$ is bounded, then $y$ and $z$ both
are bounded.
\end{theorem}

\noindent Note: Here $f$ be an unbounded modulus such that there is a
positive constant $c$ such that $f(xy) \geq cf(x)f(y)$ for all $x \geq 0, y
\geq 0$ and $\displaystyle\lim_{u \rightarrow \infty} \frac{f(u)}{u} > 0$.

\begin{proof} From the previous theorem, we have $x$ is strongly $\lambda$-summable to $L$ implies $x$ is $f_\lambda$-statistically convergent to $L$. Choose a strictly increasing sequence of positive integers $M_1 < M_2  <M_3...$ such that,

$$\frac{1}{f(\lambda_n)} f\left(\left\vert \left\{ t \in I_n : \vert x_t -L\vert \geq \frac{1}{d} \right\} \right\vert \right) < \frac{1}{d}, \mbox{~for~} n > N_d, \mbox{~where~} N_0=0$$

If $N_0< t< N_1$, let us set $z_t=0$ and $y_t=x_t$. \\

Let $d \geq 1$ and $N_d < t \leq N_{d+1}$, we set,

\[ y_t = \left\{ \begin{array}{ll}
x_t, ~z_t=0 & \mbox{if $\vert x_t -L \vert <\frac{1}{d}$};\\
L, z_t=x_t -L & \mbox{if $\vert x_t -L \vert \geq \frac{1}{d}$}.\end{array} \right. \]

Clearly, $x=y+z$ and $y$ and $z$ are bounded, if $x$ is bounded.\\

Observe that  $\vert y_t -L \vert < \xi$ for $t > N_d$ and $\vert y_t-L\vert =\vert x_t - L\vert <\xi$ if $\vert x_t -L \vert < \frac{1}{d}$.

Which follows that $\vert y_t -L \vert = \vert L-L\vert =0$ if $\vert x_t-L \vert > \frac{1}{d}$.\\

Hence, for $\xi$ arbitrary, we get $\displaystyle\lim_t y_t=L$.\\

Next, we observe that, $f(\vert \{ t \in I_n: z_t \neq 0 \} \vert ) \geq f ( \vert \{ t \in I_n : \vert z_t \vert \geq \xi \} \vert )$, for any natural number $n$ and $\xi>0$. Hence,

$$\displaystyle\lim_n \frac{f(\vert \{ t \in I_n: z_t \neq 0 \} \vert )}{f(\lambda_n)}=0;$$

that is $z$ is $f_\lambda$-statistically null.\\

We now show that if $\beta>0$ and $d \in N$ such that $\frac{1}{d}< \beta$, then $f(\vert \{ t \in I_n: z_t \neq 0 \}\vert ) < \beta$ for  all $n >N_d$. \\

If $N_d <t \leq N_{d+1}$, then $z_t \neq 0$ only if $\vert x_t -L \vert > \frac{1}{d}$. It follows that if $N_p < t \leq N_{p+1}$, then,

$$\{ t \in I_n: z_t \neq 0 \} \subseteq \left\{ t \in I_n : \vert x_t -L \vert > \frac{1}{p} \right\}.$$

Consequently,

$$\frac{1}{f(\lambda_n)}f(\vert \{ t \in I_n : z_t \neq 0 \} \vert ) \leq \frac{1}{f(\lambda_n)} f\left(\left\vert \left\{ t \in I_n: \vert x_t -L\vert > \frac{1}{p} \right\} \right\vert \right) < \frac{1}{p} < \frac{1}{d}< \beta,$$

if $N_p < n \leq N_{p+1}$ and $p > d$.\\

That is, $\displaystyle\lim_n \frac{1}{f(\lambda_n)} f(\vert \{ t \in I_n: z_t \neq 0 \} \vert )=0$.\\

This completes the proof of the theorem.

\end{proof}

\begin{corollary}
Let $f, g$ be unbounded moduli, $X$ a normed space, $(x_t)$ a sequence in $X$
and $x, y \in X$. We have

\begin{enumerate}
\item A $f_\lambda$-statistical convergent sequence is also statistically
convergent to the same limit.

\item The $f_\lambda$-statistical limit is unique whenever it exists.

\item Let $f_\lambda$-stat$\displaystyle\lim x_t = x$ and $f_\lambda$-stat$%
\displaystyle\lim y_t = y$, then,

\begin{itemize}
\item $f_\lambda$-stat$\displaystyle\lim (x_t \pm y_t) = x \pm y$

\item $f_\lambda$-stat$\displaystyle\lim \alpha x_t = \alpha x, \alpha \in R.
$
\end{itemize}

\item If $f$-st$\displaystyle\lim x_t =x$ and $g$-st$\displaystyle\lim x_t =y
$ then $x=y$, i.e two different methods of statistical convergence are
always compatible.
\end{enumerate}
\end{corollary}

\begin{theorem}
Let $(x_t)$ be a sequence in a normed space $X$. For an unbounded modulus $f$%
, $(x_t)$ is $f_\lambda$-statistically convergent to $x$ if and only if
there exists $T \subseteq \mathbb{N}$ such that $f_\lambda$-density of $T$
is zero and $(x_t)$ has $\lambda$-summable to $x$.
\end{theorem}

\begin{proof} Let us assume that  $V_z = \left\{ t \in \mathbb{N}: \displaystyle\sum_t  \Vert x_t-x \Vert > \frac{1}{z} \right\}$, for every $z \in \mathbb{N}$ such that $V_z \subset V_{z+1}$ and $\displaystyle\lim_n \frac{f(\vert V_z(n) \vert)}{f(\lambda_n)} =0$. \\

Also by induction , we obtain $\frac{f(\vert V_z(n) \vert)}{f(\lambda_n)} \leq \frac{1}{z}$ whenever $n \geq i_z$. \\

Note: Here $i_1 < i_2 < i_3...$ is a strictly increasing sequence of positive numbers such that $i_z \in V_z$.\\

Again set $T= \displaystyle\cup_{z \in \mathbb{N}} ([i_z, i_{z+1}) \cap V_z)$. Then for every $n \geq i_1$ there exists $z \in \mathbb{N}$ such that $i_z \leq n \leq i_{z+1}$ and if $m \in T(n)$ then $m < i_{z+1}$, which implies $m \in V_z$. Therefore $T(n) \subseteq V_z(n)$ and thus,

$$\frac{f(\vert T(n) \vert )}{f(\lambda_n)} \leq \frac{f(\vert V_z(n) \vert )}{f(\lambda_n)} \leq \frac{1}{z}$$

which follows that $T$ has $f_\lambda$-density zero.\\

For $\xi > 0$ and $z \in \mathbb{N}$ such that $\frac{1}{z} < \xi$, we have $n \in \mathbb{N} \backslash T$ and $n \geq i_z$ for which  there exists $q \geq z$ with $i_q \leq n \leq i_{q+1}$ and this implies $n \notin V_q$, so,

$$ \frac{1}{\lambda_n}\displaystyle\sum_t \Vert x_t -x \Vert  \leq \frac{1}{q} \leq \frac{1}{z} < \xi.$$

which shows that $(x_t)$ has $\lambda$-summable to $x$.\\

Conversely, let us assume that for $T \in \mathbb{N}$, $\displaystyle\lim_{t \in \mathbb{N} \backslash T}  \frac{1}{\lambda_n} \Vert x_t-x \Vert=0$ for $n \in I_n$ and $T$ has ${f_\lambda}$-density zero. For $\xi > 0$, there exists $t_0 \in \mathbb{N}$ such that if $t > t_0$ and $t \in \mathbb{N} \backslash T$ then

$$ \frac{1}{\lambda_n}\displaystyle\sum_t \Vert x_t -x \Vert \leq \xi.$$

This implies,

$$\left\{t \in \mathbb{N}: \displaystyle\sum_t \Vert x_t -x \Vert > \xi \right\} \subseteq T \cup \{1,..t\}$$

and then $\delta_{f_\lambda}\left(\left\{t \in \mathbb{N}: \displaystyle\sum_t \Vert x_t -x \Vert  > \xi \right\}\right)=0$.\\

\end{proof}

\vskip 0.3 cm

\begin{definition}
The sequence $(x_t)$ is $f_\lambda$-statistically Cauchy if for every $\xi >
0$ there exists $Q \in \mathbb{N}$ such that $\delta_{f_\lambda}\left(\left%
\{t \in \mathbb{N} : \displaystyle\sum_t \Vert x_t - x_Q \Vert > \xi
\right\}\right)= 0$.
\end{definition}

\begin{corollary}
A sequence is $f_\lambda$-statistically convergent implies that it is $%
f_\lambda$-statistically Cauchy. The converge is true if the space is
complete and this result is a particular case of filter convergence.
\end{corollary}

\begin{theorem}
For the Banach space $X$ where $f$ be an unbounded modulus, $f_\lambda$%
-statistically Cauchy sequence is $f_\lambda$-statistically convergent
sequence.
\end{theorem}

\begin{proof} For every $k \in N$, let $x_{Q_k}$ be such that

$$\delta_{f_\lambda}\left(\left\{t \in \mathbb{N}:  \displaystyle\sum_t \Vert x_t -x_{Q_k} \Vert > \frac{1}{k} \right\}\right)=0.$$

Consider the set $J_k = \cap_{i \leq k} B(x_{Q_i}, \frac{1}{i})$, then  for each $k \in \mathbb{N}$ we have $\mbox{diam}(J_k) \leq \frac{2}{k}$ \\

and $R_k=\{ t \in \mathbb{N}: x_t \notin J_k\}$ so that $R_k=\cup_{i \leq k} \left\{t \in \mathbb{N}: \displaystyle\sum_t \Vert x_t -x_{Q_i} \Vert > \frac{1}{i} \right\}$ which implies $f_\lambda$-density of  $R_k$ is zero. \\

Following the proof of the previous theorem, we get for $t \geq r_k$ then $\frac{f(\vert R_k(t)\vert )}{f(\lambda_t)} \leq \frac{1}{k}, r_k\in R_k$.  \\

Considering $D= \cup_{k \in \mathbb{N}} ([r_k, r_{k+1}) \cap R_k)$, we get $f_\lambda$-density of $D$ is zero.\\

Since $\cap_{k \in \mathbb{N}} J_k$ has exactly one element, say $x$, due to the completeness of $X$. We have to prove that $\displaystyle\lim_{t \in \mathbb{N} \backslash D} \frac{1}{\lambda_n} \Vert x_t -x \Vert=0$.\\

For $\xi > 0$, let us choose $i \in \mathbb{N}$ such that  $\frac{2}{i} < \xi$. If $t \geq r_j$ and $t \in \mathbb{N} \backslash D$ then there exists $k \geq i$ such that $r_k \leq t < r_{k+1}$ and then $t \notin R_k$, which implies $x_t \in J_k$ and thus,

$$ \frac{1}{\lambda_n} \displaystyle\sum_t  \Vert x_t - x \Vert  \leq \frac{2}{k} \leq \frac{2}{i} < \xi.$$

This completes the proof of the theorem.\\
\end{proof}

\begin{corollary}
A sequence which is $f_\lambda$-statistically convergent for each module $f$%
, is also convergent in ordinary sense, i.e. if $B \subseteq \mathbb{N}$ is
infinite, then there exists an unbounded module $f$ such that $f_\lambda$%
-density of $B$ is one. \newline
\end{corollary}

\begin{theorem}
Let $(x_t)$ be a sequence in $X$. If for every unbounded modulus $f$ there
exists $f_\lambda$-st$\displaystyle\lim x_t$ then all these limits are the
same $x \in X$ and $(x_t)$ is also $\lambda$-summable to $x$ in the norm
topology.\newline
\end{theorem}

\begin{proof}  It is proved that for $f,g$ two unbounded moduli, the $f_\lambda$-statistical limit is unique whenever it exists. \\

Let $X$ a normed space, $(x_t)$ a sequence in $X$ and $P, Q \in X$ and if the uniqueness is false that $\displaystyle\lim x_t =P$, there exists $\xi > 0$ such that

$$B=\left\{ t \in \mathbb{N}:  \displaystyle\sum_p \Vert x_t -P \Vert > \xi \right\}$$

is infinite.\\

Now, by choosing an unbounded modulus $f$ which will satisfy $\delta_{f_\lambda}(B)=1$, then this clearly contradicts the assumption that $f_\lambda$-stlim $x_t=x$, which completes the proof of the Theorem.\\
\end{proof}

\begin{theorem}
If $\lambda \in \Lambda$ with $\displaystyle\lim_{n \rightarrow \infty} inf
\frac{n}{f(\lambda_n)}>0$ , then $f_\lambda$-statistically convergent
sequences are $f$-statistical convergent.
\end{theorem}

\begin{proof} For given $\xi>0$, we have,

$$\{ t \leq n : \vert x_t -L \vert \geq \xi \} \supset \{ t \in I_n: \vert x_t -L \vert \geq \xi \}.$$

Therefore,

\begin{eqnarray*}
\frac{1}{f(\lambda_n)} f(\vert \{ t \leq n :\vert x_t - L \vert \geq \xi \} \vert ) &\geq& \frac{1}{f(\lambda_n)} f(\vert \{ t \in I_n :\vert x_t -L \vert \geq \xi \}\vert)\\
&\geq& \frac{f(n)}{n} \frac{n}{f(\lambda_n)} \frac{1}{f(n)} f(\vert \{ t \in I_n: \vert x_t -L \vert \geq \xi \} \vert )
\end{eqnarray*}

Taking the limit as $n \rightarrow \infty$ and using the fact that $\displaystyle\lim_n \frac{f(n)}{n} >0$, we get $(x_t)$ is $f_\lambda$-statistical convergent implies $(x_t)$ is $f$-statistical convergent.
\end{proof}

\begin{theorem}
Let $\lambda\in \Lambda$ be such that $\displaystyle\lim_n \frac{n}{%
f(\lambda_n)} =1$. Then, $f$-statistical convergence is $f_\lambda$%
-statistical convergence.
\end{theorem}

\begin{proof} Let $\beta > 0$ be given. Since $\displaystyle\lim_n \frac{n}{f(\lambda_n)}=1$, we can choose $m \in N$ such that, $\frac{f(n-\lambda_n+1)}{f(n)} < \frac{\beta}{2}$ for all $n \geq m$.\\

Now observe that,

\begin{eqnarray*}
\frac{1}{f(n)} f(\vert \{ t \leq n: \vert x_t -L \vert  \geq \xi\} \vert )
&=& \frac{1}{f(n)} f (\vert \{ t < n-\lambda_n+1 : \vert x_t -L \vert \geq \xi \} \vert ) \\
&+& \frac{1}{f(n)} f (\vert \{ t \in I_n : \vert x_t -L \vert \geq \xi \} \vert )\\
&<& \frac{f(n -\lambda_n +1)}{f(n)} + \frac{1}{f(n)} f(\vert \{ t \in I_n : \vert x_t -L \vert \geq \xi \} \vert )\\
&<& \frac{\beta}{2} + \frac{1}{f(\lambda_n)} f ( \vert \{ t \in I_n : \vert x_t - L \vert \geq \xi \} \vert )
\end{eqnarray*}

Taking limit as $n \rightarrow \infty$, we get, $(x_t)$ is $f$-statistically convergent is $f_\lambda$-statistically convergent.
\end{proof}

\begin{theorem}
Let $\lambda=(\lambda_n)$ and $\mu=(\mu_n)$ be two sequences in $\Lambda$
such that $\lambda_n \leq \mu_n$ for all $n \in \mathbb{N}_{n_0}$.

\begin{enumerate}
\item If $\lim \displaystyle\inf_{n \rightarrow \infty} \frac{\lambda_n}{%
\mu_n} > 0$ then $S_{f_\mu} \subseteq S_{f_\lambda}$.

\item If $\displaystyle\lim_{n \rightarrow \infty} \frac{\mu_n}{\lambda_n} =1
$ then $S_{f_\lambda} \subseteq S_{f_\mu}$.
\end{enumerate}
\end{theorem}

\begin{proof} 1. Suppose that $\lambda_n \leq \mu_n$ for all $n \in \mathbb{N}_{n_0}$ which implies that $f(\lambda_n) \leq f(\mu_n)$.\\

Given $\lim\displaystyle\inf_{n \rightarrow \infty} \frac{\lambda_n}{\mu_n} > 0$, then with the property of modulus function, we get $\lim\displaystyle\inf_{n \rightarrow \infty} \frac{f(\lambda_n)}{f(\mu_n)} >0$. Also, since $\lambda_n \leq \mu_n$ for all $n \in \mathbb{N}_{n_0}$, so $I_n \subset J_n$ where $I_n=[n - \lambda_n+1,n]$ and $J_n=[n-\mu_n+1, n ]$.\\

Now, for $\xi>0$, we can write,

$$\{t \in J_n: \vert x_t-L \vert \geq \xi\} \supset \{t \in I_n:\vert x_t - L \vert \geq \xi \}$$

and so,\\

$\frac{1}{\mu_n} \vert \{ t \in J_n:\vert x_t-L \vert \geq \xi \}\vert \geq \frac{\lambda_n}{\mu_n} \frac{1}{\lambda_n}\vert \{ t \in I_n: \vert x_t -L \vert \geq \xi\} \vert$, for all $n \in \mathbb{N}_{n_0}$. \\

Now, by using the definition of modulus function and satisfying $\lim \displaystyle\inf_{n \rightarrow \infty} \frac{\lambda_n}{\mu_n} > 0$ and also taking limit as $n \rightarrow \infty$, we get $S_{f_\mu}\subseteq S_{f_\lambda}$.
\end{proof}

\begin{proof} 2. Let $S_{f_\lambda}-\lim x_t =x$ and $\displaystyle\lim_{n \rightarrow \infty} \frac{\mu_n}{\lambda_n} =1$. \\

Since $I_n \subset J_n$, for $\xi >0$, we write,

\begin{eqnarray*}
\frac{1}{\mu_n} \vert \{ t \in J_n:\vert x_t -L \vert \geq \xi \} \vert &=& \frac{1}{\mu_n} \vert \{ n -\mu_n +1 \leq k \leq n-\lambda_n : \vert x_t -L \vert \geq \xi \} \vert + \frac{1}{\mu_n} \vert \{ t \in I_n:\vert x_t -L \vert \geq \xi \} \vert\\
&\leq& \frac{\mu_n -\lambda_n}{\mu_n} + \frac{1}{\mu_n} \vert \{ t \in I_n: \vert x_t -L \vert \geq \xi \} \vert\\
&\leq& \frac{\mu_n -\lambda_n}{\lambda_n} + \frac{1}{\mu_n} \vert \{ t \in I_n: \vert x_t -L \vert \geq \xi \} \vert\\
&\leq& \left(\frac{\mu_n}{\lambda_n} -1 \right) + \frac{1}{\lambda_n} \vert \{ t \in I_n: \vert x_t -L \vert \geq \xi \} \vert, \mbox{~for all ~} n \in \mathbb{N}_{n_0}.
\end{eqnarray*}

Since $\displaystyle\lim_{n \rightarrow \infty} \frac{\mu_n}{\lambda_n} =1$ and using the definition of modulus function, we can say that $x=(x_t)$ is $S_{f_\mu}$-statistically convergent is $S_{f_\lambda}$-statistcially convergent sequence, so $S_{f_\lambda} \subseteq S_{f_\mu}$.
\end{proof}

\begin{corollary}
Let $\lambda=(\lambda_n)$ and $\mu=(\mu_n)$ be two sequences in $\Lambda$
such that $\lambda_n \leq \mu_n$ for all $n \in \mathbb{N}_{n_0}$. If $\lim%
\displaystyle\inf_{n \rightarrow \infty} \frac{\lambda_n}{\mu_n} >0$, then $%
w_{f_\mu} \subset w_{f_\lambda}$.
\end{corollary}

\begin{theorem}
Let $\lambda_n \leq \mu_n$ for all $n \in \mathbb{N}_{n_0}$, then, if $\lim %
\displaystyle\inf_{n \rightarrow\infty} \frac{\lambda_n}{\mu_n} >0$, then a
strongly $w_{f_\mu}$-summable sequence is $S_{f_\lambda}$-statistically
convergent sequence.
\end{theorem}

\begin{proof} For any $\xi>0$, we have,

\begin{eqnarray*}
\displaystyle\sum_{t \in J_n} f(\vert x_t -L \vert)&=& \displaystyle\sum_{t \in J_n, f(\vert x_t -L \vert ) \geq \xi} f(\vert x_t -L \vert) +  \displaystyle\sum_{t \in J_n, f(\vert x_t -L \vert ) < \xi} f(\vert x_t -L \vert)\\
&\geq& \displaystyle\sum_{t \in I_n, f(\vert x_t -L \vert ) \geq \xi} f(\vert x_t -L \vert) + \displaystyle\sum_{t \in I_n, f(\vert x_t -L \vert ) < \xi} f(\vert x_t -L \vert)\\
&\geq& \displaystyle\sum_{t \in I_n, f(\vert x_t -L \vert ) \geq \xi} f(\vert x_t -L \vert)\\
&\geq& f\left(\displaystyle\sum_{t \in J_n, f(\vert x_t -L \vert ) \geq \xi} \vert x_t -L \vert\right)\\
&\geq& f(\vert\{ t \in I_n: \vert x_t -L \vert \geq \xi \}\vert.\xi )\\
&\geq& c f(\vert\{t \in I_n: \vert x_t -L \vert \geq \xi \} \vert). f(\xi)
\end{eqnarray*}

and so

\begin{eqnarray*}
\frac{1}{\mu_n} \displaystyle\sum_{t \in J_n} f(\vert x_t-L \vert) &\geq& \frac{1}{\mu_n} f(\vert \{ t \in I_n: \vert x_t - L \vert \geq \xi \} \vert) f(\xi)\\
&\geq& \frac{\lambda_n}{\mu_n} \frac{1}{\lambda_n} f(\vert \{ t \in I_n: \vert x_t -L \vert \geq \xi \} \vert ) f(\xi)
\end{eqnarray*}

Since $\lim \displaystyle\inf_{n \rightarrow \infty} \frac{\lambda_n}{\mu_n} >0$ and by using the definition of modulus function, we get, $(x_t)$ is strongly $f_\lambda$-summable to $L$ implies $(x_t)$  is $S_{f_\lambda}$statistically convergent to $L$.\\

This completes the proof of the theorem.

\end{proof}


\begin{thebibliography}{99}

\bibitem{Aizpuru} A. Aizpuru, M.C. List$\acute{a}$n-Garc$\acute{i}$ and F.
Rambla-Barreno, Density by Moduli and Statistical Convergence, 
Quaestiones Mathematicae 2014: 1-6.
http://dx.doi.org/10.2989/16073606.2014.894682

\bibitem{Tripathy} B. C. Tripathy and S. Borgohain , Statistically
Convergent Difference Sequence spaces of Fuzzy Real numbers defined by
Orlicz Function, Thai Journal of Mathematics , 11(2) (2013),
357-370.

\bibitem{Kolk} E. Kolk, The statistical convergence in Banach spaces, \textit%
Acta Et Commentationes Univ. Tartuensis, 928(1991), 41-52.

\bibitem{Kolk1} E. Kolk, Matrix summability of statistically convergent
sequences, Analysis, 13(1993), 77-83.

\bibitem{savas1} E. Sava$\c{s}$, On generalize ddouble statistical convergence in
locally solid Riesz spaces, Miskolc Mathematical Notes, preprint.

\bibitem{savas2} E. Sava$\c{s}$ and S. Borgohain , On strongly almost
lacunary statistical $A$-convergence defined by a Musielak-Orlicz function.
(accepted in Filomat)

\bibitem{savas3} E. Sava$\c{s}$ and S. Borgohain, Some new spaces of
lacunary $f$ -statistical $A$-convergent sequences of order $\alpha$,
 Advancements in Mathematical Sciences, AIP Conf. Proc. 1676, 2015,
020086-1\^{a}\euro ``020086-8; doi: 10.1063/1.4930512

\bibitem{savas4} E. Sava$\c{s}$, Generalized statistical convergence in
random 2-normed space, Iran. J. Sci. Technol. Trans. A Sci,
36(2012), 417-423.

\bibitem{savas5} E. Sava$\c{s}$, $I_\lambda$-statistically convergent
sequences in topological groups, Acta Et Commentationes
Universitatis Tartuensis De Mathematica, 18(1)(2014), 33-38.

\bibitem{savas8} E. Sava\c{s},  On generalized sequence spaces via modulus function, J. Inequal. Appl.,2014,  2014:101, 8 pp. 40A30.

\bibitem{savas9} E. Sava\c{s}, On some new sequence spaces defined by infinite matrix and modulus, Adv. Difference Equ., 2013, 2013:274, 9 pp.

\bibitem{savas10} E. Sava\c{s} and R. F. Patterson,  Double sequence spaces defined by a modulus. Math. Slovaca, 61(2), 245-256(2011).

\bibitem{savas11} E. Sava\c{s} and R. F. Patterson,  Lacunary statistical convergence of multiple sequences, Appl. Math. Lett., 19(6), 527-534(2006).

\bibitem{savas12} E. Sava\c{s}, On some generalized sequence spaces defined by a modulus, Indian J. Pure Appl. Math. 30(5), 459-464(1999).

\bibitem{savas13} E. Sava\c{s}, Strong almost convergence and almost $\lambda$ -statistical convergence, Hokkaido Math. J., 29(3), 531-536(2000).

\bibitem{Ca} H. Cakalli and E. Sava\c{s},  Statistical convergence of double sequences in topological groups, J. Comput. Anal. Appl., 12(2), 421-426(2010).

\bibitem{Miler} H. I. Miller and C. Orhan, On almost convergent and
statistically convergent subsequences, Acta Math. Hungar.,
93(1-2)(2001), 135-151.

\bibitem{Fast} H. Fast, Sur la convergence statistique,  Colloq.
Math. 2(1951), 241-244.

\bibitem{Nakano} H. Nakano, Concave modulars, J. Math. Soc. Japan.,
5 (1953), 29\^{a}\euro ``49.

\bibitem{Ba} I. Bala,  V. K. Bhardwaj and E. Sava\c{s},  Some sequence spaces defined by $\vert A \vert$  summability and a modulus function in seminormed space, Thai J. Math., 11(3), 623-631(2013).

\bibitem{Maddox} I. J. Maddox , Sequence Spaces Defined by a modulus, Mathematical Proceedings of the Cambridge Philosophical Society, 100
(1986), 161-166.

\bibitem{Maddox1} I.J. Maddox, Inclusion between FK spaces and Kuttner\^{a}'s theorem, Math. Proc. Camb. Philos. Soc., 101(1987),
523-527.

\bibitem{Sc} I.J. Schoenberg, The integrability of certain functions
and related summability methods, Amer. Math. Monthly, 66(1959),
361-375.

\bibitem{Connor} J. Connor and E. Sava$\c{s}$, Lacunary statistical and
sliding window convergence for measurable functions, Acta
Mathematica Hungarica, 145(2)(2015), 416-432.

\bibitem{Fridy} J.A. Fridy, On statistical convergence, Analysis, 5(1985),
301-313.

\bibitem{Mursaleen} M. Mursaleen, $\lambda$-statistical convergence, 
Math. Slovaca, 50(2000), 111-115.

\bibitem{Pa} Richard F. Patterson and  E. Sava\c{s},  Lacunary statistical convergence of double sequences, Math. Commun., 10(1), 55-61(2005).

\bibitem{Sat} T. $\check{S}$al$\grave{a}$t, On statistically convergent
sequences of fuzzy real numbers, Math. Slovaca, 30(1980), 139-150.

\bibitem{Ruckle} W. H. Ruckle, FK spaces in which the sequence of coordinate
vectors is bounded, Canad. J. Math., 25 (1973), 973\^{a}\euro ``978.




\end{thebibliography}
\end{document}